\documentclass[12pt, leqno,amscd, psamsfonts]{amsart}
\usepackage{amsthm, amsmath}
\usepackage{amssymb} 
\usepackage{amscd}
\usepackage[all]{xy}

\DeclareFontEncoding{OT2}{}{} 

\def\arrow#1#2{\smash{\mathop{\longrightarrow}\limits^{#1}_{#2}}} 
\def\arrowdown#1#2{\Big\downarrow \rlap{$\vcenter{\hbox{$\scriptstyle#2$}}$}
{\hbox to -10pt{\hss{$\vcenter{\hbox{$\scriptstyle#1$}}$}}}}
\def\arrowup#1#2{\Big\uparrow \rlap{$\vcenter{\hbox{$\scriptstyle#2$}}$}
{\hbox to -10pt{\hss{$\vcenter{\hbox{$\scriptstyle#1$}}$}}}}

\newtheorem{thm}{Theorem}[section]
\newtheorem{prop}[thm]{Proposition}
\newtheorem{lemma}[thm]{Lemma}

\newtheorem{Definition}[thm]{Definition}
\newtheorem{Hypothesis}[thm]{Hypothesis}

\newtheorem{Example}[thm]{Example}

\newtheorem{cor}[thm]{Corollary}

\newcounter{ex}[section]

\numberwithin{equation}{section}
\numberwithin{thm}{section}

 \renewcommand{\O}{{\mathcal O}}

\newcommand{\K}{\overline{k}}

\def\thfill{\null\nobreak\hfill}

\def\endproof{\thfill\vbox{\hrule
  \hbox{\vrule\hbox to 5pt{\vbox to 5pt{\vfil}\hfil}\vrule}\hrule}}

\begin{document}
\title[Derived invariants]{Derived category invariants and L-series}

\author[Ph. Cassou-Nogu\`es]{Ph. Cassou-Nogu\`es }
\address{Philippe Cassou-Nogu\`es, IMB\\ Univ. Bordeaux 1\\  33405 Talence, France.\\}
 
 \email{Philippe.Cassou-Nogues@math.u-bordeaux1.fr}

\author[T. Chinburg]{T. Chinburg*}\thanks{*Supported by NSF Grant \#DMS0801030}
\address{Ted Chinburg, Dept. of Math\\Univ. of Penn.\\Phila. PA. 19104, U.S.A.}
\email{ted@math.upenn.edu}

\author[B. Erez]{B. Erez }
\address{Boas Erez, IMB, Univ. Bordeaux 1, 33405 Talence, France.\\}

 \email{Boas.Erez@math.u-bordeaux1.fr}

\author[M. Taylor]{M. J. Taylor}
\address{Martin J. Taylor, Merton College \\ 
Oxford OX1 4JD, U.K.}
\email{martin.taylor@merton.ox.ac.uk}

\date{\today}

\begin{abstract}
We relate invariants in derived categories associated to
tame actions of finite groups on projective varieties over a finite field
to zeros of L-functions
\end{abstract}

\maketitle

\section{Introduction}
\label{s:intro}


A recurring theme in the study of values of L-functions of arithmetic
schemes is that these should be related to Euler characteristics
of various kinds.   Behind the cohomology groups needed to define
such Euler characteristics are hypercohomology complexes
in derived categories.  In this paper we consider how to determine
the additional information contained in such complexes beyond what
is seen by Euler characteristics.  In the geometric situations we
consider, this additional information takes the form of extension classes.
Our main result is that two natural extension classes constructed from
\'etale and coherent cohomology differ by a numerical invariant
which is the reciprocal of the zero of an $L$-function.  This 
suggests that it may be fruitful to study relationships between
derived category invariants and L-functions in more general
contexts, e.g. for projective schemes over $\mathbb{Z}$.

We will now describe the contents  of this paper.
In Section 2  we consider the following geometric situation.  Let $X$ be a smooth
projective variety over the algebraic closure $\K$ of a finite field $k$
having a tame, generically free action over $\K$ of a finite group $G$.
We suppose that the cohomology groups of $\O_X$ vanish except in
dimensions $0$ and $n$ for some integer $n > 0$ and that
the characteristic $p$ of $k$ divides $\# G$.  In this
case, the isomorphism class of the hypercohomology complex $H^\bullet(X,\O_X)$
in the derived category of the homotopy category of $\K[ G]$-modules
is determined by its Euler characteristic and an extension class
$\beta(X,G)$ in the one-dimensional $\K$-vector space 
$\mathrm{Ext}^{n+1}_{\K[ G]}(H^n(X,\O_X),H^0(X,\O_X))$. 
Similarly, the isomorphism class of $H^\bullet_{et}(X,k)$ in the derived
category of the homotopy category of complexes of $k[G]$-modules is determined
by its Euler characteristic together with an extension class
$\gamma(X,G)$ in the one-dimensional $k$-vector space 
$\mathrm{Ext}^{n+1}_{k [G]}(H^n_{et}(X,k),H^0_{et}(X,k))$.

In Theorem 2.9    we show that $$\beta(X,G) = 1 \otimes \gamma(X,G)$$
relative to a natural isomorphism
$$\K \otimes_k \mathrm{Ext}^{n+1}_{k [G]}(H^n_{et}(X,k),H^0_{et}(X,k)) = \mathrm{Ext}^{n+1}_{\K [G]}(H^n(X,\O_X),H^0(X,\O_X))$$
induced by  the map $k \to \mathbb{G}_a$ of \'etale sheaves on $X$.  One can thus think of $$k\cdot \beta(X,G) = k \cdot \gamma(X,G)$$ as the ``\'etale $k$-line" inside
$\mathrm{Ext}^{n+1}_{\K [G]}(H^n(X,\O_X),H^0(X,\O_X))$.  

In Section 3 we make some additional hypotheses on $G$ and $X$ described in Hypothesis \ref{hyp:dimension2}. We assume in particular that the $p$-Sylow subgroups of $G$ are cyclic and non trivial. We let $C$ be a $p$-Sylow subgroup of $G$.  Under these hypotheses  we  define a $k$-line $$k \cdot \alpha(X,G)$$
inside $
\mathrm{Ext}^{n+1}_{\K [C]}(H^n(X,\O_X),H^0(X,\O_X))$ which is 
associated to a model $Y_0$ over $k$ of the quotient $Y = X/G$
of $X$ by the action of $G$.   One can think
of $k \cdot \alpha(X,G)$ as a $k$-line determined by 
coordinates for the one-dimensional $\K$-vector
space $\mathrm{Ext}^{n+1}_{\K[ C]}(H^n(X,\O_X),H^0(X,\O_X))$
which arises  from the model $Y_0$. 

The restriction map induces an isomorphism of $\K$-vector spaces 
$$\mathrm{Ext}^{n+1}_{\K [G]}(H^n(X,\O_X),H^0(X,\O_X))\to \mathrm{Ext}^{n+1}_{\K [C]}(H^n(X,\O_X),H^0(X,\O_X)). $$
We consider $k.\beta(X,G)$ as a $k$-line of $\mathrm{Ext}^{n+1}_{\K [C]}(H^n(X,\O_X),H^0(X,\O_X))$ via this isomorphism. Our goal is to compare 
in $\mathrm{Ext}^{n+1}_{\K [C]}(H^n(X,\O_X),H^0(X,\O_X))$ the \'etale $k$-line with the $k$-line provided by the model $Y_0$. To be more precise 
our main result, Theorem
\ref{thm:secondclass} and its corollaries, is that 
$$k \cdot \alpha(X,G) = \zeta \cdot k \cdot \beta(X,G)$$
for a constant $\zeta \in \K^*$ such that
$$\mu(X,G) = \zeta^{1-\#k} \in  k^*$$
is independent of all choices and is the reciprocal of
a zero of an $L$-function associated to $X$.  If $n$
is odd, this $L$-function is  the numerator of
the mod $p$ zeta function of $Y_0$ over $k$.  If $n$ is even, the
$L$-function is the denominator of the mod $p$ zeta function
of the variety $X_0$ which is the quotient of $X$ by 
the group  generated by a lift to $X$ of the arithmetic
Frobenius in $\mathrm{Gal}(Y/Y_0)$.   If $X$ is an elliptic curve
and $k = \mathbb{Z}/p$ then 
$\mu(X,G)$ is simply the Hasse invariant associated to
$Y_0$ (c.f. Example \ref{ex:ellcurves}). 

In the last section of this paper we provide examples of projective varieties  of arbitrary large dimension, endowed with an  action of a cyclic group of order $p$  
for which the  hypotheses of  Theorem \ref{thm:secondclass} are fulfilled. 
\medbreak

\noindent \thanks{{\bf Acknowledgments:} The first author would like to thank the University of  Pennsylvania  for its hospitality during work on this paper.}
 
\section{Varieties with two non-vanishing cohomology groups}
\label{s:hypo}

Let $k$ be a finite field of order $q = p^f$, where $p$
is a prime, and let $\K$ be an algebraic closure of $k$.
We will suppose that $G$ is a finite group of order divisible by $p$ acting
tamely and generically freely over $\K$ on a smooth projective variety $X$ 
over $\K$ of dimension $d$.  Let $\pi:X \to Y= X/G$ be
the quotient morphism.  
If $\mathcal{F}$ is a coherent $G$-sheaf
on $X$, we denote $H^i(X,\mathcal{F})$ by $H^i(\mathcal{F})$.

\begin{Hypothesis}
\label{hyp:dimension}
There is an integer $n \ge 1$ such that $H^i(\O_X) \ne \{0\}$
if and only if $i \in \{0,n\}$.
\end{Hypothesis}

\begin{lemma}
\label{lem:coherent}
The coherent hypercohomology complex $H^\bullet(\O_X)$ is isomorphic in the derived category $D(\K G)$ of the homotopy category of complexes of $\K [G]$-modules to
a perfect complex $P^\bullet$ of $\K [G]$-modules which has
trivial terms outside degrees in the interval $[0,n]$.  This
complex defines an exact sequence
\begin{equation}
\label{eq:exact}
0 \to H^0(\O_X) \to P_0 \to \cdots \to P_n \to H^n(\O_X) \to 0
\end{equation}
and thereby an extension class $\beta(X,G)$ in
$\mathrm{Ext}^{n+1}_{\K [G]}(H^n(\O_X),H^0(\O_X))$. 
\end{lemma}

\begin{proof}  By a result of Nakajima \cite{Nak}, $H^\bullet(\O_X)$
is isomorphic to a perfect complex in $D(\K G)$ because
the action of $G$ on $X$ is tame.  Because of
hypothesis \ref{hyp:dimension}, we can truncate this complex to
arrive at $P^\bullet$.  
\end{proof}

\begin{Definition}
\label{def:frobmod}  Let $F:\K \to \K$ be the arithmetic Frobenius automorphism over
$\K$, so that $F(\alpha) = \alpha^q$ for $\alpha \in \K$.   A $k$-linear map $T:M_1 \to M_2$ between vector spaces over $\K$ will be called 
semilinear  (resp. anti-semilinear) if $T(\alpha \cdot m_1) = F(\alpha) T(m_1) $ (resp. $T(\alpha \cdot m_1) = F^{-1}(\alpha) T(m_1) $) for $\alpha \in \K$
and $m_1 \in M_1$.  
\end{Definition}

\begin{lemma}
\label{lem:extcomp}
 Suppose that $\ell$ is a field of characteristic $p$, and that there is 
 an exact sequence of $\ell G$-modules 
 \begin{equation}
 \label{eq:exactP}
 0 \to \ell \to P_0 \to \cdots \to P_n \to M \to 0
 \end{equation}
  in which $P_i$ is projective and finitely generated for all $i$.  Then 
  $\mathrm{Ext}^{n+1}_{\ell [G]}(M,\ell)$ is a one-dimensional
  $\ell$ vector space with respect to the multiplication action
  of $\ell$ on $\ell$.  The degeneration of the spectral sequence
$H^p(G,\mathrm{Ext}_{\ell}^q(M,\ell)) \rightarrow \mathrm{Ext}^{p+q}_{\ell [G]}(M,\ell)$
gives an isomorphism
\begin{equation}
\label{eq:degenit}
\mathrm{Ext}^{n+1}_{\ell [G]}(M,\ell) = H^{n+1}(G,\mathrm{Hom}_\ell(M,\ell))
\end{equation}
\end{lemma}

\begin{proof}
By dimension shifting via the sequence (\ref{eq:exactP}),
we get an exact sequence
$$\mathrm{Hom}_{\ell [G]}(P_0,\ell) \to \mathrm{Hom}_{\ell [G]}(\ell,\ell) \to \mathrm{Ext}^{n+1}_{\ell [G]}(M,\ell) \to 0.$$  Here $ \mathrm{Hom}_{\ell [G]}(\ell,\ell) = \ell$.  So either $\mathrm{Ext}^{n+1}_{\ell [G]}(M,\ell)$ is a one-dimensional $\ell$-vector space,
or the injection $\ell  \to P_0$ splits.  However, we have assumed
$p$ divides the order of $G$, so $\ell$ is not a projective $\ell [G]$-module and 
the latter alternative is impossible.      
\end{proof}

\begin{cor}
\label{cor:action}
Suppose $\ell  = \K$ and that $T:M \to M$ is a semilinear map commuting with the action of
$G$ on $M$.  There is a $G$-equivariant anti-semilinear endomorphism $T^{-1}$  of $\mathrm{Hom}_{\K}(M,\K)$ defined
by $$T^{-1}(f)(m) = F^{-1}(f(T(m)))\quad \mathrm{for}\quad f \in \mathrm{Hom}_{\K}(M,\K)\quad \mathrm{and}\quad m \in M.$$   Via (\ref{eq:degenit}) this gives an anti-semilinear action 
of $T^{-1}$ on $\mathrm{Ext}^{n+1}_{\K [G]}(M,\K) \cong \K$.
  \end{cor}

 The following result is Lemma III.4.13
 of \cite{Milne};  see also \cite[\S XXII.1]{Groth} and  \cite[p. 143]{Mumford}.
 
 \begin{lemma}
 \label{lem:Milne}  Let $V$ be a finite dimensional vector space
 over $\K$, and let $\phi:V \to V$ be a semilinear map.  
 Then $V$ decomposes as
 a direct sum $V = V_s \oplus V_\eta$, where $V_s$ and $V_\eta$ are
 subspaces stable under $\phi$, $\phi$ is bijective on $V_s$ and
 $\phi$ is nilpotent on $V_\eta$.  Moreover, $V_s$ has a basis
 $\{e_1,\ldots,e_t\}$ such that $\phi(e_i) = e_i$ for all $i$.
 It follows that $V^\phi$ is the $k$-vector space having 
 basis $\{e_1,\ldots,e_t\}$ and $\phi - 1:V \to V$ is surjective.
 \end{lemma}

We have
an exact Artin-Schreier sequence of \'etale sheaves on $X$ given by
\begin{equation}
\label{eq:AS}
0 \arrow{}{} k \arrow{}{} \mathbb{G}_a \arrow{F-1}{} \mathbb{G}_a \to 0
\end{equation}
where $F:\mathbb{G}_a \to \mathbb{G}_a$ is the arithmetic Frobenius
morphism defined by $\alpha \mapsto \alpha^q = F(\alpha)$ for $\alpha$ a local section of $\mathbb{G}_a$.  By \cite[Remark III.3.8]{Milne}, there is an isomorphism
$H^\bullet(X,\O_X) \to H^\bullet_{et}(X,\mathbb{G}_a)$ in the derived category,
giving isomorphisms $H^j(\O_X) \to H^j_{et}(X,\mathbb{G}_a)$ for all $j$.

\begin{lemma}
\label{lem:splitting}
The long exact cohomology sequence associated to (\ref{eq:AS}) splits into
short exact sequences
\begin{equation}
\label{eq:exacter}
0 \arrow{}{} H^i_{et}(X,k) \arrow{}{} H^i(\O_X) \arrow{F-1}{} H^i(\O_X) \arrow{}{}
0
\end{equation}
for all $i$. The terms of this sequence are trivial if $i \not \in \{0,n\}$.  When $i = 0$,
one has $H^0_{et}(X,k) = k$ and $H^0(\O_X) = \K$.  When $i = n$, there is a 
$\K G$-module isomorphism
\begin{equation}
\label{eq:dirsum}
H^{n}(\O_X) = H^n(\O_X)_{F,s} \oplus H^n(\O_X)_{F,\eta}
\end{equation}
arising from Lemma \ref{lem:Milne} in which 
$F$ is an isomorphism on $ H^i(\O_X)_{F,s}$ and nilpotent on $H^i(\O_X)_{F,\eta}$.  
The sequence (\ref{eq:exacter}) with $i = n$ shows $H^n(\O_X)^F = H^n_{et}(X,k)$ and $H^n(\O_X)_{F,s} = \K \otimes_k H^n_{et}(X,k)$.
\end{lemma}

\begin{proof}
The action of $F$ on $H^i_{et}(X,\mathbb{G}_a) = H^i(\O_X)$ is semilinear for
all $i$.  The split exact sequences (\ref{eq:exacter}) arise from the fact that by Lemma \ref{lem:Milne},
$$F-1: H^i(\O_X) \to H^i(\O_X)$$ is surjective for all $i$.  When $i = n$,
the decomposition in (\ref{eq:dirsum}) is a $\K G$-module decomposition because
$F$ commutes with the action of $G$, $$H^n(\O_X)_{F,s} = \cap_{m \ge 1} F^m(H^n(\O_X))$$ and $$H^n(\O_X)_{F,\eta} = \mathrm{Kernel}(F^m:H^n(\O_X) \to H^n(\O_X))\quad \mathrm{if}\quad m >> 0.$$  The sequence (\ref{eq:exacter}) shows 
$H^n(\O_X)^F = H^n_{et}(X,k)$ so $H^n(\O_X)_{F,s} = \K \otimes_k H^n_{et}(X,k)$
by Lemma \ref{lem:Milne}
\end{proof}

\begin{lemma}
\label{lem:hypocon}
 The complex $H^\bullet_{et}(X,k)$ is perfect,
 and $H^j_{et}(X,k) \ne 0$ if and only if $j \in \{0,n\}$.  
 The sequence (\ref{eq:AS}) gives   rise to a morphism  
\begin{equation}
\label{eq:into}
H^\bullet_{et}(X,k) \to H^\bullet_{et}(X,\mathbb{G}_a) = H^\bullet(\O_X)
\end{equation}
 in the derived category of complexes of $k [G]$-modules.  Let $H^\bullet(\O_X)'$
be the complex resulting from  $H^\bullet(\O_X)$ by truncating $H^\bullet(\O_X)$ in dimensions greater than $n$ and by replacing $H^n(\O_X)$ by the submodule
 $H^n(\O_X)_{F,s}$ appearing in (\ref{eq:dirsum}).  The morphism
 (\ref{eq:into}) gives a quasi-isomorphism
 \begin{equation}
 \label{eq:zipper}
 \K \otimes_{L,k} H^\bullet_{et}(X,k) = H^\bullet(\O_X)'
 \end{equation}
 of perfect complexes of $\K [G]$-modules, where $L$ on the left is the left derived tensor product.
 The $\K [G]$-module $H^n(\O_X)_{F,\eta}$ in (\ref{eq:dirsum}) is projective.
 \end{lemma}
 
 \begin{proof}
  Since $X \to Y$ is a tame $G$-cover, the sheaf $\pi_* k$ in the \'etale topology on $Y$
  is a sheaf of projective $k[G]$-modules.  The argument of \cite[Theorem VI.13.11]{Milne} now shows that $H^\bullet_{et}(X,k)$ is a perfect complex of $k[G]$-modules.   The isomorphism (\ref{eq:zipper})
  in the derived category results form the calculation of the cohomology 
  groups $H^i_{et}(X,k)$ in Lemma \ref{lem:splitting}, where the left derived
  tensor product $\K \otimes_{L,k}$ is just the tensor product because all $k$-modules
  are free.   This implies $H^\bullet(\O_X)'$ is perfect because $H^\bullet_{et}(X,k)$
  is.  Because $H^\bullet(\O_X)$ is also perfect, we conclude that 
  $H^n(\O_X)_{F,\eta}$ must be projective because
 this was the module truncated from $H^\bullet(\O_X)$ in degree $n$ in
 the construction of $H^\bullet(\O_X)'$.  
 \end{proof}
 \vskip 0.1 truecm

 It follows from Lemma 2.8 that $H^{\bullet}_{et}(X,k)$ is a perfect complex such that  $H^i_{et}(X, k) \ne \{0\}$
if and only if $i \in \{0,n\}$. Following the lines of Lemmas  2.2 and 2.4 we conclude that we can attach to $H^\bullet_{et}(X,k)$ an extension class $\gamma(X, G)$ in 
$\mathrm{Ext}^{n+1}_{k [G]}(H^n_{et}(X,k),H^0_{et}(X,k))$ and prove that this $k$-vector space is of dimension $1$.

 \begin{thm}
 \label{thm:extcomps}  
   The morphism
 (\ref{eq:into}) leads to an isomorphism of one-dimensional $\K$ vector spaces 
 \begin{equation}
 \label{eq:leadto}
 \K \otimes_k \mathrm{Ext}^{n+1}_{k [G]}(H^n_{et}(X,k),H^0_{et}(X,k)) = 
 \mathrm{Ext}^{n+1}_{\K[ G]}(H^n(\O_X),H^0(\O_X))
 \end{equation}
 such that 
 \begin{equation}
 \label{eq:tenisom}
 \beta(X,G) = 1 \otimes \gamma(X,G) .
 \end{equation}
  The action of $F$ on $H^n(\O_X)$
 and on $H^0(\O_X) = \K$ leads to an anti-semilinear action of $F^{-1}$ on 
 $\mathrm{Ext}^{n+1}_{\K [G]}(H^n(\O_X),H^0(\O_X))$. 
  Via (\ref{eq:leadto}), the one-dimensional $k$-vector space
$L_1 = 1 \otimes_k \mathrm{Ext}^{n+1}_{k [G]}(H^n_{et}(X,k),H^0_{et}(X,k))$ is the subspace
of $\mathrm{Ext}^{n+1}_{\K G}(H^n(\O_X),H^0(\O_X))$ which is fixed by $F^{-1}$.
In particular, $\beta(X,G)$ is fixed by $F^{-1}$.
 \end{thm}
 
 \begin{proof}
 Since
$H^n(\O_X)_{F,\eta}$ is a  projective $\K [G]$-module by Lemma \ref{lem:hypocon},  inclusion $H^n(\O_X)_{F,s} \to H^n(\O_X)$  induces an isomorphism of one-dimensional $\K$-vector spaces
\begin{equation}
\label{eq:firstisom}
\mathrm{Ext}^{n+1}_{\K [G]}(H^n(\O_X),H^0(\O_X)) \to 
\mathrm{Ext}^{n+1}_{\K [G]}(H^n(\O_X)_{F,s},H^0(\O_X)) 
\end{equation}
which sends the extension class $\beta(X,G)$ associated
to $H^\bullet(\O_X)$ to the extension class $\beta(X,G)'$
associated to $H^\bullet(\O_X)'$.  In view of Lemma \ref{lem:hypocon}, the isomorphism
$$\K \otimes_{L,k} H^\bullet_{et}(X,k) \cong H^\bullet(\O_X)'$$ in the derived
category gives an isomorphism
 \begin{equation}
 \label{eq:leadinto}
 \K \otimes_k \mathrm{Ext}^{n+1}_{k [G]}(H^n_{et}(X,k),H^0_{et}(X,k)) = 
 \mathrm{Ext}^{n+1}_{\K [G]}(H^n(\O_X)_{F,s},H^0(\O_X))
 \end{equation}
 of one-dimensional vector spaces over $\K$ which identities $1 \otimes \gamma(X,G)$ with $\beta(X,G)'$
 when $\gamma(X,G)$ is the extension class in 
 $\mathrm{Ext}^{n+1}_{k [G]}(H^n_{et}(X,k),H^0_{et}(X,k))$ associated
 to $H^\bullet_{et}(X,k)$.   Combining
 (\ref{eq:firstisom}) and (\ref{eq:leadinto}) thus leads to an isomorphism
 (\ref{eq:leadto}) which identifies $\beta(X,G)$ with $1 \otimes \gamma(X,G)$. 
 
 The action of $F$ on $H^n(\O_X)$ is via the action of $F$ on $\O_X$
 and commutes with the action of $G$.  (This $F$ is different from the
 $\K$-linear relative Frobenius automorphism $F_{X/\K}$ of $H^n(\O_X) = H^n_{et}(X,\mathbb{G}_a)$ described by Milne in \cite[\S VI.13]{Milne}.)  
 Since $F$ acts semilinearly and fixes both $H^n_{et}(X,k) \subset H^n(\O_X)$
 and $H^0_{et}(X,k) = k \subset \K = H^0(\O_X)$,
 the remaining assertions in Theorem  \ref{thm:extcomps} follow from
 (\ref{eq:leadto}).
 \end{proof}

\section{Extension class invariants arising from models}
\label{s:next}
 In this section we will assume the   following  strengthening of Hypothesis 2.1. .
\begin{Hypothesis}
\label{hyp:dimension2}
The $p$-Sylow subgroups of the group $G$ are cyclic and non trivial and the $\K$-vector spaces $H^n(G, \bar k)$ and $ H^{n+1}(G, \bar k)$ are of dimension one. The 
variety $X$ is of dimension $n$ and  $H^i(\O_X) = \{0\}$
if and only if $i \not \in \{0,n\}$.  There exists 
 a smooth projective variety $Y_0$ over $k$ for which the following is true.
\begin{enumerate}
\item[a.] $Y = X/G = \K \otimes_k Y_0$.
\item[b.] The morphism 
$\tilde \pi:X \to Y_0$ which is the composition of $\pi:X \to Y$ with the projection
$Y \to Y_0$ is Galois.  
\end{enumerate}
\end{Hypothesis}
We fix once for all a $p$-Sylow subgroup $C$ of $G$. Since $\K$ is of characteristic $p$,  for any integer $m$, the restriction map induces an injection
\begin{equation}
\label{eq:resmap}
\mbox{Res}^G_C: H^m(G, \bar k)\mapsto H^m(C, \bar k).
\end{equation}
Since we have assumed $C$ to be cyclic and non-trivial,  the $\bar k$-vector spaces $H^m(C, \bar k)$ are of dimension $1$ for all $m$. Therefore, 
Hypothesis \ref{hyp:dimension2} requires that (\ref{eq:resmap}) is
an isomorphism for $m \in \{n,n+1\}$.

\begin{Example}
\label{ex:semid}
Suppose that $G$ is the semi-direct product of a normal subgroup $H$
of order prime to $p$ with a non-trivial cyclic $p$-group $C$.  Then  the groups $H^i(H, \bar k)$ are trivial for  $i\geq 0$. Therefore the inflation
 homomorphisms
 $$\mbox{Inf}: H^i(G/H, \bar k)\rightarrow H^i(G, \bar k)$$
 are  isomorphisms  and
 $$H^i(C, \bar k)\simeq H^i(G, \bar k),\  \mbox{for}\ i\geq 0 . $$
 We conclude that, in this case,  the dimensions of $H^n(G, \bar k)$ and $H^{n+1}(G, \bar k)$  are both  equal to one as required in Hypothesis \ref{hyp:dimension2}.
 \end{Example}
 
 The aim of this section is to show  that the model $Y_0$ for $Y$ over $k$ leads  to a class
$$\alpha(X,G) \in \mathrm{Ext}^{n+1}_{\K [C]}(H^n(\O_X),H^0(\O_X))$$
which is well defined up to multiplication by an element of $k^*$, and which is different in general from the class obtained by restriction from the class $\beta(X,G)$ constructed
in the previous section.  This new class should be understood as an obstruction to a descent problem,  and to be more precise,   the  descent of  the action $X\times G\longrightarrow X$ defined over $\K$ to an action $X_0\times G\rightarrow X_0$  defined over $k$. The key to constructing $\alpha(X,G)$
is given by the following three results.

\begin{prop}
\label{prop:tryitall}Let $\Gamma = \mathrm{Gal}(X/Y_0)$.  The 
morphism of sheaves on $Y_0$ given by $\O_{Y_0} \to (\tilde \pi)_* \O_X$
leads to a homomorphism 
$$H^n(Y_0,\O_{Y_0}) \to H^n(Y_0,(\tilde \pi)_* \O_X) = H^n(\O_X)$$
and an exact sequence
\begin{equation}
\label{eq:zerogood0}
0 \arrow{}{} W \arrow{}{} H^n(Y_0,\O_{Y_0}) \arrow{}{}  H^n(\O_X)^\Gamma \arrow{}{} W' \arrow{}{} 0
\end{equation}
in which $W$ and $W'$ are $k$ vector spaces of dimension $1$ with a trivial
action of $G$.  Tensoring $\O_{Y_0} \to (\tilde \pi)_* \O_X$ with $\K$ over $k$
gives the natural homomorphism $\O_Y \to \pi_* \O_X$ of sheaves on $Y$.  Tensoring (\ref{eq:zerogood0}) with $\K$ over $k$ gives the exact sequence
\begin{equation}
\label{eq:zerogood1}
0 \arrow{}{} \K \otimes_k W \arrow{}{} H^n(Y,\O_{Y}) \arrow{}{}  H^n(\O_X)^G \arrow{}{} \K \otimes_k W' \arrow{}{} 0.
\end{equation}
associated to the homomorphism $H^n(Y,\O_{Y}) \to H^n(Y,\pi_* \O_X) = H^n(\O_X)$
which results from $\O_Y \to \pi_* \O_X$.
\end{prop}

\begin {prop}
\label{prop:propnodd} Suppose that $n$ is odd.   Then the trace map $Tr_*: H^n(\mathcal{O}_X)\rightarrow H^n(\mathcal{O}_Y)$  and the
inclusion $ \K\otimes_k W \rightarrow H^{n}(\mathcal{O}_{Y})$ induce $\bar k$-linear maps
$$Ext_{\bar k[C]} ^{n+1}(H^{n}(\mathcal{O}_{Y}), \bar k)\rightarrow Ext_{\bar k[C]} ^{n+1}(H^{n}(\mathcal{O}_{X}), \bar k)\ $$
and
$$Ext_{\bar k[C]} ^{n+1}(H^{n}(\mathcal{O}_{Y}), \bar k)\rightarrow Ext_{\bar k[C]} ^{n+1}(\K\otimes_k W, \bar k) $$
respectively. These maps are both surjective with the same kernel, giving an isomorphism
\begin{equation}
\label{eq:noddcase} Ext_{\bar k[C]} ^{n+1}(H^{n}(\mathcal{O}_{X}), \bar k)\rightarrow Ext_{\bar k[C]} ^{n+1}(\K\otimes_k W, \bar k)=\K\otimes_k Ext_{ k[C]} ^{n+1}( W,  k) . \end{equation}
\end{prop}

\begin{prop}
\label{prop:propneven} Suppose $n$ is even.  
Then the  inclusion $H^n(\mathcal{O}_X)^G \to H^n(\mathcal{O}_X)$ and the surjection
$H^n(\mathcal{O}_X)^G \to \K\otimes_k W'$ coming from (3.3)   lead to $\bar k$-vector
space maps
$$ \mbox{Ext}^{n+1}_{\bar k [C]}(H^n(\mathcal{O}_X),H^0(\mathcal{O}_X)) \rightarrow  \mbox{Ext}^{n+1}_{\bar k [C]}(H^n(\mathcal{O}_X)^G,H^0(\mathcal{O}_X))$$ and
$$\mbox{Ext}^{n+1}_{\bar k [C]}( \K\otimes_k W',H^0(\mathcal{O}_X))  \rightarrow  \mbox{Ext}^{n+1}_{\bar k [C]}(H^n(\mathcal{O}_X)^G,H^0(\mathcal{O}_X))$$
respectively.
These maps are both injective with the same image, leading to an isomorphism
\begin{equation}
\label{eq:nevencase}\mbox{Ext}^{n+1}_{\bar k [C]}(H^n(\mathcal{O}_X),H^0(\mathcal{O}_X)) = \mbox{Ext}^{n+1}_{\bar k [C]}( \K\otimes_k W', \K)
=\K\otimes_k  \mbox{Ext}^{n+1}_{k [C]}( W', k)  . \end{equation}

\end{prop}
We will give the proof of these Propositions in the next section.  

Our aim is now to use  these  propositions to construct $\alpha(X,G)$ and a numerical
invariant $\mu(X,G)$. The restriction map induces an injective homomorphism of $\K$-vector spaces 
$$\mathrm{Ext}^{n+1}_{\K [G]}(H^n(\O_X),H^0(\O_X))\to \mathrm{Ext}^{n+1}_{\K [C]}(H^n(\O_X),H^0(\O_X)) .$$
Since these vector spaces are of dimension $1$ this is an isomorphism. We identify in what follows the class $\beta(X, G)$ and  the $k$-line $L_1$ defined in Section 2  with 
their images  in 
$\mathrm{Ext}^{n+1}_{\K [C]}(H^n(\O_X),H^0(\O_X))$. 

\begin{thm}
\label{thm:secondclass}  Let $L_0$ be the $k$-line in 
$\mathrm{Ext}^{n+1}_{\K [C]}(H^n(\O_X),H^0(\O_X))$ which is the image
of $1 \otimes \mathrm{Ext}^{n+1}_{k [C]}(W,k)$ (resp.
$\mathrm{Ext}^{n+1}_{k [C]}(W',k)$) under the isomorphism
in (\ref{eq:noddcase}) (resp. (\ref{eq:nevencase})) if $n$ is odd (resp. if
$n$ is even).  
Let $\alpha(X,G)$ be any generator of $L_0$ over $k$, so that $\alpha(X,G)$
is defined only up to multiplication by an element of $k^*$.  
Then \begin{equation}
\label{eq:alphadef}
\alpha(X,G) = \zeta \cdot \beta(X,G)
\end{equation}
for an element  $\zeta \in \K^*$ which is well-defined up to multiplication by an element of $k^*$.  The constant 
\begin{equation}
\label{eq:invdef}
\mu(X,G) = \zeta^{1-q} \in  \K^*
\end{equation} 
lies in $\K^*$ and is an invariant of the action of $G$ on $X$. 
\end{thm}

\begin{proof}  This follows from Propositions \ref{prop:tryitall}, \ref{prop:propnodd} and
\ref{prop:propneven}
and the fact that $\mathrm{Ext}^{n+1}_{\K [C]}(H^n(\O_X),H^0(\O_X))$ has dimension
$1$ over $\K$.
\end{proof}

\begin{cor}
\label{lem:hush}
 
Let $F^{-1}$ be the anti-semilinear endomorphism of $\mathrm{Ext}_{\K [C]}(H^n(\O_X),H^0(\O_X))$
induced by the action of $F$ on $H^n(\O_X)$ and the action of $F^{-1}$
on $H^0(\O_X) = \K$,  described as in Theorem 2.9. Then $F^{-1}$ acts on $L_0$ by multiplication by $\mu(X,G)$.  The constant
$\mu(X,G)$ lies in $k^*$.
\end{cor}
\begin{proof} Let $L_1$ be the $k$-line of $\mathrm{Ext}_{\K [C]}^{n+1}(H^n(\O_X),H^0(\O_X))$ introduced in Theorem 2.9. It follows from the definitions of Theorem 3.6  that 
$$L_0=\zeta L_1$$ for any $\zeta$ satisfying (3.6). Suppose that $c_1$ is a $k$-basis of $L_1$ so that $\zeta.c_1$ is a $k$-basis of $L_0$. The endomorphism $F^{-1}$ fixes 
$L_1$ by Theorem 2.9. Hence since $F^{-1}$ is anti-semilinear, we have 
$$F^{-1}(\tau \cdot \zeta  \cdot c_1) = \tau\cdot F^{-1}(\zeta) \cdot c_1 = \nu  \cdot (\tau \cdot \zeta \cdot c_1)$$
for $\tau \in k$, where
$$\nu = \frac{F^{-1}(\zeta)}{\zeta}.$$
  This proves that $F^{-1}$ acts as multiplication by $\nu$ on the $k$-line
  $L_0$, so $\nu \in k^*$.  Hence
  $$\nu = F(\nu) = \frac{\zeta}{F(\zeta)} = \zeta^{1 - q}=\mu(X, G)$$
  which completes the proof in view of (\ref{eq:invdef}).
\end{proof}
\begin{cor}
\label{cor:roots}
The action of  $F$ on $\O_{Y_0}$ and on $\O_X$ induces a $k$-linear
action on $W$ and $W'$.   If $n$ is odd (resp. even) then $F$ acts on
$W$ (resp. $W'$) by multiplication by $\mu(X,G) \in k^*$.
\end{cor}
\begin{proof}  The action of $F$ on $\O_{Y_0}$ and on $\O_X$ is
via the map $\alpha \to \alpha^q$ on local sections, and is $k$-linear.
Thus $F$ respects the homomorphism $\O_{Y_0} \to (\tilde \pi)_* \O_X$
in Proposition \ref{prop:tryitall}, so it follows that $F$ acts $k$-linearly
on the one dimensional $k$-vector spaces $W$ and $W'$. Suppose
$n$ is odd.  Since $C$ is cyclic
there are isomorphisms
\begin{eqnarray}
\mathrm{Ext}_{k [C]}^{n+1}(W,k) &= & H^{n+1}(C,\mathrm{Hom}_k(W,k))\nonumber\\
&=& \hat{H}^0(C,\mathrm{Hom}_k(W,k)) \\
&=& \mathrm{Hom}_{k[C]}(W,k)/\mathrm{Tr}_C \mathrm{Hom}_k(W,k)\nonumber
\end{eqnarray}
Because $W \cong k$ with trivial $C$-action, this gives an $F^{-1}$-equivariant
isomorphism between $\mathrm{Ext}_{k[C]}^{n+1}(W,k)$ and $\mathrm{Hom}_k(W,k)$.
Recall that $F^{-1}$ sends an element $f \in \mathrm{Hom}_k(W,k)$ to the homorphism
$F^{-1} f$ defined by $(F^{-1} f)(w) = F^{-1}(f(F(w)))$ for $w \in W$.  Since $\mathrm{dim}_k W = 1$ the action of $F$ on $W$ is given by multiplication by some $\alpha \in k$, and $F$
fixed $k$.  Hence
$$(F^{-1} f)(w) = f(\alpha w) = \alpha f(w)$$
so $F^{-1}$ acts on $\mathrm{Hom}_k(W,k)$ by multiplication by $\alpha$.  Thus
$\alpha$ is also the eigenvalue of $F^{-1}$ on $L_0=\mathrm{Ext}_{k[G]}^{n+1}(W,k)$,
so $\alpha= \mu(X,G)$ by Corollary 3.8.  The proof when $n$ is
even 
is similar.
\end{proof}

We now relate $\mu(X,G)$ to zeta functions.  
Let $\zeta(V/k,T)$ be the zeta function of a smooth projective variety 
$V$ over $k$.  Then $\zeta(V/k,T) \in \mathbb{Z}_p[[T]]$, and the 
congruence formula in \cite[Expos\'e XXII, 3.1]{Groth} is
\begin{equation}
\label{eq:congform}
{\zeta}(V/k,T) =  \prod_{i = 0}^{\mathrm{dim}(V)} \mathrm{det}(1 - FT| H^i(V,\O_V))^{(-1)^{i+1}} \quad \mathrm{mod}\quad p\mathbb{Z}_p[[T]].
\end{equation}
  Write  this formula as
\begin{equation}
\label{eq:numdum}
{\zeta}(V/k,T)\quad \mathrm{mod} \quad p \mathbb{Z}_p[[T]] = 
\frac{\zeta_1(V/k,T)}{\zeta_0(V/k,T)} 
\end{equation}
where 
\begin{equation}
\label{eq:evodd}
\zeta_j(V/k,T) = \prod_{i \equiv j \quad \mathrm{mod} \quad 2} \mathrm{det}(1 - FT| H^i(V,\O_V))
\end{equation}
for $j = 0, 1$.  Note that $\zeta_1(V/k,T)$ and $\zeta_0(V/k,T)$ could have
a common zero, so the formula (\ref{eq:numdum}) does not imply that a 
zero of $\zeta_1(V/k,T)$ (resp. of $\zeta_0(V/k,T)$) is a zero (resp. pole)
of $\zeta(V/k,T) $ mod $p\mathbb{Z}_p[[T]]$.

\begin{cor}
\label{cor:answer}
  If $n$ is odd then
 $\mu(X,G)^{-1}$ is a zero of ${\zeta}_1(Y_0/k,T)$.  Suppose $n$ is
 even.  Let $X_0$ be the quotient of $X$ by the the action of a lift $\phi_X$
 to $\mathrm{Gal}(X/Y_0)$ of the arithmetic Frobenius of $\mathrm{Gal}(Y/Y_0) \equiv \mathrm{Gal}(\K/k)$.
 Then $X_0$ is a smooth projective variety over $k$ such that $ X = \K \otimes_k X_0$, and $\mu(X,G)^{-1}$
is a zero of ${\zeta}_0(X_0/k,T)$.  
\end{cor}

\begin{proof}
If $n$ is odd, we have shown $\mu(X,G)$ is the eigenvalue of $F$ acting
on the one-dimensional $k$-space $W$ inside $H^n(Y_0,\O_{Y_0})$.  
Hence $1 - \mu(X,G)^{-1} F$ is not invertible on $H^n(Y_0,\O_{Y_0})$,
so $\mu(X,G)^{-1}$ is a zero of $\zeta_1(Y_0/k,T)$. Suppose $n$ is even.
We have shown $\mu(X,G)$ is the eigenvalue of $F$ acting
on a one-dimensional $k$-space $W'$ which is a quotient
of $H^n(\O_X)^\Gamma$, where $\Gamma = \mathrm{Gal}(X/Y_0)$.
It is shown in Lemma \ref{lem:conshyp} below that $\Gamma$ is the semidirect product of the normal subgroup $G$ with the closure
$\overline {\langle \phi_X \rangle}$ of the subgroup generated by
$\phi_X$.  Since $X \to Y_0$ is a pro-\'etale cover, it follows that $X_0$ is smooth and projective over $k$,
and that $X = \K \otimes_k X_0$.  Hence $H^n(\O_X) = \K \otimes_k H^n(X_0,\O_{X_0})$,
so $H^n(\O_X)^{\overline {\langle \phi_X \rangle}} = H^n(X_0,\O_{X_0})$.
 Thus $W'$ is a subquotient of $H^n(X_0,\O_{X_0})$, so
as above, $1 - \mu(X,G)^{-1} F$ is not invertible on $H^n(X_0,\O_{X_0})$.
Since we assumed $n$ is even, this shows $\mu(X,G)$ is a zero
of $\zeta_0(X_0/k,T)$.
\end{proof}
\begin{Example}
\label{ex:curves} {\rm Suppose that  $X$ is a curve and $n = 1$ in Hypothesis
\ref{hyp:dimension2}.  Then
$$\zeta(Y_0,T) = \frac{P_1(Y_0,T)}{(1 - T)(1 - qT)}$$
when $P_1(Y_0,T)\in \mathbb{Z}[T]$ is the characteristic polynomial of Frobenius acting
on $H^1_{et}(Y_0,\mathbb{Q}_\ell)$ for any prime $\ell$ different from $p$.
One has $\zeta_0(Y_0,T) = (1 - T)$ since $F$ acts trivially on $H^0(Y_0,\O_{Y_0}) = k$
and $H^j(Y_0,\O_{Y_0}) = 0$ if $j > 1$.  Since $(1 - q T)^{-1} \equiv 1$ mod $p \mathbb{Z}_p[[T]]$, we conclude that
\begin{equation}
\label{eq:curveanswer}
 P_1(Y_0,T) \equiv \zeta_1(Y_0,T)  \quad \mathrm{mod} \quad p\mathbb{Z}_p[[T]]
\end{equation}
Thus Corollary \ref{cor:answer} shows $\mu(X,G)^{-1}$ is a zero in $k$ of $P_1(Y_0,T)$.}
\end{Example}

\begin{Example}
\label{ex:ellcurves} {\rm Suppose that in example \ref{ex:curves}, $X$ is an elliptic curve. 
Since $\pi:X \to Y$ is an \'etale $G = \mathbb{Z}/p$ cover, $Y$ must be an
ordinary elliptic curve, and $Y_0$ has genus $1$ over $k$.  Because $Y_0$
has a point defined over $k$ by the Weil bound, $Y_0$ is isomorphic to an elliptic curve over $k$, and $Y_0$ is ordinary.  Since $\mathrm{dim}_k H^1(Y_0,\O_{Y_0}) = 1$,
we see that $\mu(X,G)^{-1}$ is the unique zero of $P_1(Y_0,T)$ in $k$.
Thus $\mu(X,G)$ is the image in $k$ of the unit root of Frobenius acting on 
$H^1_{et}(Y_0,\mathbb{Q}_\ell)$.  Suppose now that $k = \mathbb{Z}/p$,
and let $\underline{0}$ be the origin of $Y_0$, so that $\underline{0}$
has residue field $k$.  Let $t$ be a uniformizing parameter in the local
ring $\O_{Y_0,\underline{0}}$.  There is a unique differential $\omega \in H^0(Y_0,\Omega^1{Y_0/k})$ having an expansion 
$$\omega = \sum_{\nu = 1}^\infty c_\nu t^{\nu -1} dt$$
at $\underline{0}$ for which $c_1 = 1$.  In \cite[Appendix 2, \S5]{Lang}, Lang 
defines $c_p$ to be the Hasse invariant of $Y_0$.  Lang shows that changing
$t$ to $bt$ for some $b \in \O_{Y_0}^*$ changes $c_p$ by $\overline{b}^{p-1}$
where $\overline{b}$ is the image of $b$ in the residue field $k$ of $\O_{Y_0}$.
Since we have now assumed $k = \mathbb{Z}/p$, one has $\overline{b}^{p-1} = 1$,
so $c_p = c_p(Y_0) \in k$ is independent of the choice of $t$.  
The formula in \cite[Appendix 2, \S2, Thm. 2]{Lang} shows 
$c_p(Y_0)$ is the eigenvalue of $F$ acting on $H^1(Y_0,\O_{Y_0})$.
So we conclude from the fact that $W = H^1(Y_0,\O_{Y_0})$ in Corollary \ref{cor:roots} that $\mu(X,G)$ is the Hasse invariant $c_p(Y_0)$.}
\end{Example}

\section{Proof of Propositions \ref{prop:tryitall}, \ref{prop:propnodd} and
\ref{prop:propneven}.}
\label{s:proofitall}

Throughout this section we assume Hypothesis \ref{hyp:dimension2}.

 If $\mathcal{F}$ is a  $G$-sheaf on $X$ (resp. $Y$) we denote  $H^i(X,  \mathcal {F}) $ (resp.   $H^i(Y,  \mathcal {F})) $ by $H^i(\mathcal{F})$. The
 quotient morphism  $\pi:X \to Y= X/G$ induces  a  natural morphism  $\iota: \mathcal{O}_{Y}\rightarrow \pi_{*}(\mathcal{O}_{X})$ which identifies $\mathcal{O}_{Y}$ with
 $\pi_*(\mathcal{O}_{X})^G$. Moreover, we have a morphism $ \pi_{*}(\mathcal{O}_{X})\rightarrow \mathcal{O}_{Y}$ induced by the trace element $Tr_G$ of $\K [G]$. We denote by
 $\iota_*: H^n( \mathcal{O}_{Y})\rightarrow H^n( \mathcal{O}_{X})^G $ and $Tr_*: H^n( \mathcal{O}_{X})\rightarrow H^n( \mathcal{O}_{Y})$ the morphisms of
 $\bar k [G]$-modules respectively induced by $\iota$ and $Tr_G$. We define $L$ and $L'$ by the exact sequence:
\begin{equation}
\label{eq:coker}
0\arrow{}{} L \arrow{}{} H^n(Y,\O_Y) \arrow{\iota_*}{} H^n(\O_X)^G \arrow{}{} L' \arrow{}{} 0
\end{equation}

\subsection {Proof of Proposition 3.3}\ 
 
First we need some preliminary results. 
\begin{lemma}
\label{lem:conshyp}
The constant field of $Y_0$ is $k$, and  $\mathrm{Gal}(X/Y_0)$ is the semi
direct product of $\mathrm{Gal}(Y/Y_0) \cong \mathrm{Gal}(\K /k ) \cong \hat {\mathbb{Z}}$
with the normal subgroup $G = \mathrm{Gal}(X/Y)$.  A lift $\phi_X$
to $\mathrm{Gal}(X/Y_0)$ of the arithmetic Frobenius automorphism $1 \otimes F$ of $\mathrm{Gal}(Y/Y_0)$
is well defined up to an element of $G$.  \end{lemma}

\begin{proof}  The constant field of $Y_0$ is $k$ because $Y = \K \otimes_k Y_0$
is a variety with constant field $\K$.  The exact sequence 
$$1 \to G \to \mathrm{Gal}(X/Y_0) \to \mathrm{Gal}(Y/Y_0) \to 1$$
splits because $\mathrm{Gal}(Y/Y_0)$ is pro-free on one generator.  The 
rest of the Lemma is now clear.
\end{proof}
\begin{lemma}
\label{lem:okey}
Recall that we have assumed  that $Y = \K \otimes_k Y_0$ for a smooth projective variety $Y_0$ over
$k$ and that the induced morphism $X \to Y_0$ is \'etale and Galois.  Let $\Gamma = \mathrm{Gal}(X/Y_0)$.  The flat base change isomorphism $H^n(Y,\O_Y) = \K \otimes_k H^n(Y_0,\O_{Y_0})$ gives an exact sequence
\begin{equation}
\label{eq:zerogood3}
0 \arrow{}{} W \arrow{}{} H^n(Y_0,\O_{Y_0}) \arrow{\iota_*}{}  H^n(\O_X)^\Gamma \arrow{}{} W' \arrow{}{} 0
\end{equation}
in which $W$ and $W'$ are $k$-vector spaces.   The tensor product
of (\ref{eq:zerogood3}) with $\K$ over $k$ is the exact sequence (\ref{eq:coker}).
\end{lemma}

\begin{proof}
  Recall that $\phi_X \in \Gamma$ is a choice
 of lift of the arithmetic Frobenius $\phi_Y \in \mathrm{Gal}(Y/Y_0) \cong \mathrm{Gal}(\K/k)$. By Lemma \ref{lem:conshyp}, $\Gamma$ is the semi-direct product of the normal subgroup $G$
 with the closure $\overline {\langle \phi_X \rangle}$ of the subgroup generated by $\phi_X$.  Hence $\phi_X$ acts as a 
 semi-linear automorphism of the $\K$-vector space $H^n(\O_X)^G$, and
 \begin{equation}
 \label{eq:gammainv}
 (H^n(\O_X)^G) ^{\overline {\langle \phi_X \rangle}} = H^n(\O_X)^\Gamma.
 \end{equation}
 Lemma \ref{lem:Milne} shows
 $$H^n(\O_X)^G = \K \otimes_k (H^n(\O_X)^G)^{\overline {\langle \phi_X \rangle}}.$$
 since $\phi_X$ is an automorphism of $H^n(\O_X)^G$.
 Combining this with (\ref{eq:gammainv}) proves that 
 $$H^n(\O_X)^G = \K \otimes_k H^n(\O_X)^\Gamma.$$
 Thus $\iota_*:H^n(Y,\O_Y) \to H^n(\O_X)^G$ results from tensoring
 the natural homomorphism $H^n(Y_0,\O_{Y_0}) \to H^n(\O_X)^\Gamma$
 with $\K$ over $k$.  Since tensoring with $\K$ over $k$ is exact the tensor product of (4.2) with $\K$ over $k$ is (4.1) as required.   \end{proof}

Since it follows from  Lemma 4.2  that  
$$L=\K\otimes_k W \ \ \mbox{and}\   L'=\K\otimes_k W' $$ we note that in order to complete the proof of Proposition 3.3 it suffices to show that $L$ and $L'$ are one-dimensional $\K$-vector spaces. This will be a consequence of   Hypothesis 3.1 and the next   proposition. 

\begin {prop}There exist $\bar k$-isomorphisms of vector spaces

$$ L\simeq H^n(G, \bar k), \ \ L'\simeq H^{n+1}(G, \bar k) $$
where $G$ acts trivially on $\bar k$.
\end {prop}

\noindent \proof  We start the proof by establishing a lemma.
\begin {lemma}The $\bar k$-linear map $Tr_*$ induces an isomorphism of $\bar k$-vector spaces
$$Tr_{*}:H^{n}(\mathcal{O}_{X})_{G}\rightarrow H^{n}(O_{Y}).\ $$

\noindent Moreover, the composition of $Tr_*$ with $\iota_*$  is the map
$$Tr_G: H^{n}(\mathcal{O}_{X})\rightarrow H^{n}(\mathcal{O}_{X})$$
induced by the multiplication by the trace  element $Tr_G$ of $\bar k [G]$.
\end{lemma}

\noindent \proof  The exact sequence of sheaves on $Y$
$$\{0\}\rightarrow \mathcal{L}\rightarrow \pi_{*}(\mathcal{O}_{X})\rightarrow \mathcal{O}_{Y}\rightarrow \{0\}\ $$
associated to $Tr_{G}: \pi_{*}(\mathcal{O}_{X})\rightarrow \mathcal{O}_{Y}$ induces  a long exact sequence of
$\bar k$-vector spaces
$$H^{n}(\mathcal{L})\rightarrow H^{n}(\mathcal{O}_{X})\rightarrow H^{n}(\mathcal{O}_{Y})
\rightarrow H^{n+1}(\mathcal{L})\rightarrow ...$$
Because dim$(Y)=n$,   $H^{n+1}(\mathcal{L})=0$. Therefore
$Tr_{*}: H^{n}(\mathcal{O}_{X})\rightarrow  H^{n}(\mathcal{O}_{Y})$ is onto with
kernel $M$ equal to the image  of $ H^{n}(\mathcal{L})$  in $H^n(\mathcal {O}_X)$. To prove the
Lemma, it will suffice to show that $M$ is equal to the kernel $I_GH^n(\mathcal{O}_X)$
of the surjection $H^n(\mathcal{O}_X) \to H^n(\mathcal{O}_X)_G$, where $I_G$ is the augmentation
ideal of $\mathbb{Z}[G]$.  

Since $Tr_{*} \circ (1-h)  =0$ for $h \in G$, we have $I_G  H^n(\mathcal{O}_X) \subset M$.
To show $M \subset I_G  H^n(\mathcal{O}_X)$, we first observe that $\mathcal{L}=I_{G}\pi_{*}(\mathcal{O}_{X})$ since
$\pi_* (\mathcal{O}_X)$ is a locally free rank one $\mathcal{O}_Y[G]$-module.  Let $\mathcal{E}$
be the kernel of the natural morphism
$$\oplus_{h\in G}(1-h): \oplus_{h\in G}\pi_{*}(\mathcal{O}_{X})\mapsto I_{G}\pi_* \mathcal{O}_{X}=\mathcal{L}.$$
Since $H^{n+1}(\mathcal{E})=0$,  this morphism induces a surjective map
$\tau:\oplus_{h\in G}H^{n}(\pi_{*}(\mathcal{O}_{X}))\longrightarrow{}{} H^{n}(\mathcal{L})$.  Therefore the image $M$ of
$ H^{n}(\mathcal{L})$  in $H^n(\mathcal {O}_X)$ is contained in the image of the composition of $\tau$
with the homomorphism $H^{n}(\mathcal{L})\rightarrow H^{n}(\mathcal{O}_{X})$.  The latter composition
is the map on cohomology induced by the homomorphism $\oplus_{h \in G}(1-h): \oplus_{h \in G} 
\pi_{*}(\mathcal{O}_{X})\rightarrow \pi_{*}(\mathcal{O}_{X})$.  Since $I_G$ is additively generated by $1-h$ as
$h$ ranges over $G$, this shows that $M\subset I_G  H^n(\mathcal{O}_X)$, which completes the proof. $\Box$

\bigskip
 \noindent  We  associate to   the exact sequence
\begin{equation}
\label{eq:exact}
0 \to H^0(\mathcal{O}_X) \to P_0 \to \cdots \to P_n \to H^n(\mathcal{O}_X) \to 0
\end{equation}
the following commutative diagram:

\[\xymatrix{
P_{n-1, G}\ar[r]^{\partial}\ar[d]_{Tr_G}&P_{n, G}\ar[r]^{}
\ar[d]_{Tr_G}& H^n(\mathcal{O}_X)_G\ar[r]^{}\ar[d]_{Tr} &0\\
P_{n-1}\ar[r]^{\partial}&P_n\ar[r]^{}&
H^n(\mathcal{O}_X)\ar[r]^{}&0\\
}\]
here $Tr$ denotes the composite of the isomorphism $Tr_{*}:H^{n}(\mathcal{O}_{X})_{G}\rightarrow H^{n}(O_{Y}) $ and the identification $\iota_*: H^n( \mathcal{O}_{Y})\rightarrow H^n( \mathcal{O}_{X})^G \subset{H^n( \mathcal{O}_{X})} $; thus

\[\xymatrix{
0\ar[r]^{}&\partial(P_{n-1, G})\ar[r]^{}\ar[d]_{Tr_G}&P_{n, G}\ar[r]^{}
\ar[d]_{Tr_G}& H^n(\mathcal{O}_X)_G\ar[r]^{}\ar[d]_{Tr} &0\\
0\ar[r]^{}&\partial(P_{n-1})\ar[r]^{}&P_n\ar[r]^{}&
H^n(\mathcal{O}_X)\ar[r]^{}&0.\\
}\]
Since the $\bar k[G]$-modules $P_l$ are all projective,  the map $Tr_G$ induces and isomorphism from $P_{l,G}\simeq P^G_l$. Moreover, it follows from  Lemma 4.4 
that $L$ identifies with the kernel of $Tr_*$. By using  the Snake lemma we obtain  an exact sequence:
\begin{equation}0\rightarrow L\rightarrow \frac {\partial(P_{n-1})}{\partial(P^G_{n-1})}\rightarrow\frac{P_n}{P_n^G}\rightarrow \frac{H^n(\mathcal{O}_X)}{Tr(H^n(\mathcal{O}_X))}\rightarrow 0\end{equation}
and hence an isomorphism of $\bar k$-vector spaces
\begin{equation}L\simeq \frac{\partial(P_{n-1})\cap P_n^G}{\partial(P_{n-1}^G)} .\end{equation}
Since the $\bar k[G]$-modules $P_0, ...,P_n$ are all projective, and hence injective, we can extend the complex
$$0\rightarrow P_0\rightarrow...\rightarrow P_n$$
to an injective resolution $P^{\bullet}$ of $H^0(\mathcal{O}_X)=\bar k$. It follows from (4.6)  that we have the following isomorphisms:
$$L\simeq H^n(\mbox {Hom}_{\bar k [G]}(\bar k, P^{\bullet})\simeq \mbox {Ext}^n_{\bar k[G]}(\bar k, \bar k)\simeq H^n(G, \bar k) . $$
In view of Lemma 4.4  we can identify $L'$ with  the $\bar k$-vector space $ \frac{H^n(\mathcal{O}_X)^G}{Tr_G(H^n(\mathcal{O}_X))}$. This leads us to the exact sequence
$$P_n^G\rightarrow (\frac{P_n}{\partial(P_{n-1})})^G\rightarrow L'\rightarrow 0 .$$
 Using  the complex $P^{\bullet}$ once again, we deduce from the previous sequence that
\begin{equation} L'\simeq \frac{\partial(P_n)\cap P^G_{n+1}}{\partial(P_n^G)}\simeq H^{n+1}(G, \bar k) .
\end{equation}

\subsection { Proof of Proposition 3.4}\ 

We recall that $C$ is a $p$-Sylow subgroup of $G$. 

\begin{lemma}
\label{lem:module}
Define $M(n)$ to be the $\K [C]$-module given by
$\K$ with trivial $C$-action if $n$ is odd and by the quotient  $\K [C]/(\K \cdot \mathrm{Tr}_C)$ if $n$ 
is even, where $\mathrm{Tr}_C = \sum_{\sigma \in C} \sigma$ is the trace element of $\K [C]$.  
There is a $\K [C]$-module isomorphism
$$H^n(\O_X) = M(n) \oplus M$$
in which $M$ is a finitely generated
free $\K [C]$-module.  
\end{lemma}
\begin{proof}  Since $C$ is a  $p$-group and $\K$ is of characteristic $p$, the ring $\K [C]$ is local and artinian and hence every projective $\K [C]$-module is
free and injective.  The result now follows from the existence of
the sequence (\ref{eq:exact}) together with induction on $n$.
\end{proof}
\vskip 0.1 truecm
 Now we are using the notations of Proposition 3.3. 
 We have an isomorphism of $\bar k$-vector spaces
$$Ext_{\bar k[C]} ^{n+1} (F, \bar k)\simeq H^{n+1}(C, \mbox{Hom}_{\bar k}(F, \bar k))\  $$
  for any $\bar k [C]$-module $F$, where  $\mbox{Hom}_{\bar k}(F, \bar k) $ is endowed with the  $\bar k [C]$-module structure given by
$$cf: m\rightarrow f(c^{-1}m)\ \quad \forall c\in C, \quad \forall f \in \mbox{Hom}_{\bar k}(F, \bar k) . $$
 Since $C$ is cyclic and $n$ is odd,
$$H^{n+1}(C, \mbox{Hom}_{\bar k}(F, \bar k))\simeq \mbox{Hom}_{\bar k}(F, \bar k)^C/Tr_{C}(\mbox{Hom}_{\bar k}(F, \bar k))\ .$$
This leads us to  identify $Ext_{\bar k[C]} ^{n+1}(F, \bar k)$ and $ \mbox{Hom}_{\bar k }(F, \bar k)^C/Tr_{C}(\mbox{Hom}_{\bar k}(F, \bar k))$.
Under these  identifications  the map
$$\alpha: Ext_{\bar k[C]} ^{n+1}(H^{n}(\mathcal{O}_{Y}), \bar k)\rightarrow Ext_{\bar k[C]} ^{n+1}(H^{n}(\mathcal{O}_{X}), \bar k)\ $$  is induced by
$$\mbox{Hom}_{\bar k}(H^{n}(\mathcal{O}_{Y}), \bar k)\rightarrow \mbox{Hom}_{\bar k}(H^{n}(\mathcal{O}_{X}), \bar k)^C$$

$$f\mapsto f\circ Tr_* $$
while the map  
$$\beta: Ext_{\bar k[C]} ^{n+1}(H^{n}(\mathcal{O}_{Y}), \bar k)\rightarrow Ext_{\bar k[C]} ^{n+1}(\K\otimes_k W, \bar k) $$
 is the restriction map
$$\mbox{Hom}_{\bar k}(H^{n}(\mathcal{O}_{Y}), \bar k)\rightarrow \mbox{Hom}_{\bar k}(L, \bar k) . $$
Since, by hypothesis, $L$ is a $\bar k$-sub-vector space of $H^n (\mathcal{O}_Y)$ of dimension one,  the $\bar k$-linear map $\beta$ is
clearly surjective and $\mbox{Ker}(\beta)$ is a subvector space  of $Ext_{\bar k[C]} ^{n+1}(H^{n}(\mathcal{O}_{Y}), \bar k)$
of codimension one.

\noindent Let   $f\in \mbox{Hom}_{\bar k}(H^{n}(\mathcal{O}_{Y}), \bar k)$  be an element of  $\mbox{Ker}(\alpha)$. Then   there exists
$h\in \mbox{Hom}_{\bar k}(H^n(\mathcal{O}_X, \bar k)$ such that $f\circ Tr_*=Tr_C(h)$. Let $q$ be the index of $C$ in $G$ and let
$\{\tau \in  S\}$ be a set of representatives of $G/C$. For any $x\in H^n(\mathcal{O}_X)$ we  have the equalities:
\begin{equation}
q(f\circ Tr_*(x))=f\circ Tr_*(\sum_{\tau\in S}\tau x)=\sum_{c\in C}h(c\sum_{\tau\in S}(\tau x))=h(Tr_G(x)).
\end{equation}
 It  follows  from  Lemma 4.4  that any element $a$ in $L$   can be written $a=Tr_*(x)$, with $Tr_Gx=0$. Therefore we deduce from (4.8) that for every $a\in L$
$$qf(a)= q(f\circ Tr_*(x))=h(Tr_Gx)=0 $$
and thus that $f(a)=0$.    We conclude that  $f\in \mbox{Ker}(\beta)$. Hence we have proved that $\mbox{Ker}(\alpha)$ is contained in $\mbox{Ker}(\beta)$ and so that $1\leq \mbox{codim}(\mbox
{Ker}(\alpha)$. We now observe that to complete the proof of the proposition it suffices to prove that 
$$\mbox{Hom}_{\bar k}(H^n(\mathcal{O}_X), \bar k)^C/Tr_{C}(\mbox{Hom}_{\bar k}(H^n(\mathcal{O}_X), \bar k))$$ is
a $\bar k$-vector space of dimension one. As a consequence we  will obtain that   $\mbox{codim}(\mbox
{Ker}(\alpha))\leq 1$ and we thereby deduce the equality $\mbox{Ker}(\alpha)$ and  $\mbox{Ker}(\beta)$ and hence  the surjectivity of $\alpha$.

\noindent Since $n$ is odd we deduce from  Lemma 4.5  that  there exists  a $\bar k [C]$-module isomorphism
$$H^n(\mathcal{O}_X)=\bar k\oplus M $$ where $M$ is a finitely generated free $\bar k [C]$-module. The $\bar k$-vector space $ \mbox{Hom}_{\bar k}(H^n(\mathcal{O}_X), \bar k)$ splits into a direct sum
\begin{equation}\mbox{Hom}_{\bar k}(H^n(\mathcal{O}_X), \bar k)=\mbox{Hom}_{\bar k}(\bar k, \bar k)\oplus \mbox{Hom}_{\bar k}(M, \bar k)=\bar k \oplus\mbox{Hom}_{\bar k}(M, \bar k) .\end{equation}
Therefore
\begin{equation}
\mbox{Hom}_{\bar k}(H^n(\mathcal{O}_X), \bar k)^C=\bar k\oplus \mbox{Hom}_{\bar k}(M, \bar k)^C
\end{equation}
while
\begin{equation}
Tr_C(\mbox{Hom}_{\bar k_C}(H^n(\mathcal{O}_X), \bar k))=\{0\}\oplus Tr_C(\mbox{Hom}_{\bar k}(M, \bar k).
\end{equation}
Since $M$ is  a free $\bar k[C]$-module one easily checks that $\mbox{Hom}_{\bar k}(M, \bar k)$ is  $\bar k [C]$-free and thus that
$\mbox{Hom}_{\bar k}(M, \bar k)^C=Tr_C(\mbox{Hom}_{\bar k}(M, \bar k)$. Hence we have proved that
$$\mbox{Hom}_{\bar k}(H^n(\mathcal{O}_X), \bar k)^C/Tr_C(\mbox{Hom}_{\bar k}(H^n(\mathcal{O}_X), \bar k)\simeq \bar k$$
as required. $\Box $
\vskip 0.1 truecm
\noindent {\bf Remark.}  One should note that the conclusions of Proposition 3.4 are incorrect when $n$ is even. 
\vskip 0.2 truecm 

\subsection { Proof of Proposition 3.5.}\ 

 For the sake of simplicity in the proof we identify  $H^0(\mathcal{O}_X)$ and $\bar k$ and $W'\otimes_k\K$ with $L'$. 
 
 Since $C$ is cyclic and $n$ is even  we have  isomorphisms
\begin{equation} \mbox{Ext}^{n+1}_{\bar k [C]}(L',\bar k)\simeq \hat H_{-1}(C, \mbox{Hom}_{\bar k}(L', \bar k))=\mbox{Hom}_{\bar k}(L', \bar k)_{Tr_C}/(c-1)(\mbox{Hom}_{\bar k}(L', \bar k))
=\mbox{Hom}_{\bar k}(L', \bar k) \end{equation}
where, for  a $\bar k[C]$-module $Z$,  we let $Z_{Tr_C}$ denote the submodule which is annihilated by $Tr_C$.
By a similar argument we obtain the isomorphism

\begin{equation}\mbox{Ext}^{n+1}_{\bar k [C]}(H^n(\mathcal{O}_X)^G,\bar k)\simeq \mbox{Hom}_{\bar k}(H^n(\mathcal{O}_X)^G, \bar k). \end{equation}
Since $L'=H^n(\mathcal{O}_X)^G/Tr_G(H^n(\mathcal{O}_X))$  is a $\bar k$-vector space of dimension one, it follows from (4.12) and (4.13) that
$\mbox{Ext}^{n+1}_{\bar k [C]}(L',\bar k)$ is of dimension one and that  the map 
\begin{equation}\mbox{Ext}^{n+1}_{\bar k [C]}(L',H^0(\mathcal{O}_X))  \rightarrow  \mbox{Ext}^{n+1}_{\bar k [C]}(H^n(\mathcal{O}_X)^G,H^0(\mathcal{O}_X))\end{equation} can be identified  with the homomorphism
$$\mbox{Hom}_{\bar k}(H^n(\mathcal{O}_X)^G/Tr_G(H^n(\mathcal{O}_X)), \bar k))\rightarrow \mbox{Hom}_{\bar k}(H^n(\mathcal{O}_X)^G, \bar k))$$
obtained by composing an element of $\mbox{Hom}_{\bar k}(H^n(\mathcal{O}_X)^G/Tr_G(H^n(\mathcal{O}_X)), \bar k))$ with the natural surjection
$H^n(\mathcal{O}_X)^G\rightarrow H^n(\mathcal{O}_X)^G/Tr_G(H^n(\mathcal{O}_X)$. This shows that (4.14) is an injective map which identifies
$\mbox{Ext}^{n+1}_{\bar k [C]}(L',\bar k)$ with
\begin{equation} \{f\ \in \mbox{Hom}_{\bar k}(H^n(\mathcal{O}_X)^G, \bar k))\  \mbox {such that }\ f |Tr_G(H^n(\mathcal{O}_X))=0\}.\end{equation}

Since $n$ is even, according to Lemma 4.5  we can decompose $H^n(\mathcal{O}_X)$ into a direct sum of $\bar k [C]$-modules
$$H^n(\mathcal{O}_X)=M(n)\oplus M$$
where $M(n)=\frac {\bar k[C]}{\bar kTr_C}$ and $M$ is a free $\bar k [C]$-module.
This implies   the following decompositions into direct sums of $\bar k$-vector spaces:
\begin{equation} \mbox{Ext}^{n+1}_{\bar k [C]}(H^n(\mathcal{O}_X), \bar k)=\mbox{Ext}^{n+1}_{\bar k [C]}(M(n), \bar k)\oplus \mbox{Ext}^{n+1}_{\bar k [C]}(M, \bar k)=\mbox{Ext}^{n+1}_{\bar k [C]}(M(n), \bar k)\oplus\{0\}
\end{equation}
It now follows from (4.16)  that
\begin{equation}
\mbox{Ext}^{n+1}_{\bar k [C]}(H^n(\mathcal{O}_X), \bar k)=
\mbox{Hom}_{\bar k}(\frac {\bar k[C]}{\bar kTr_C}, \bar k) /(c-1)\mbox{Hom}_{\bar k}(\frac {\bar k[C]}{\bar kTr_C}, \bar k).
\end{equation}
The dimension over  $\bar k$ of the right-hand side of this equality is the dimension of the kernel of the multiplication by $(c-1)$ on
$ \mbox{Hom}_{\bar k}(\frac {\bar k[C]}{\bar kTr_C}, \bar k)$. The kernel of the multiplication by $(c-1)$ naturally identifies  with the vector space
$\mbox{Hom}_{\bar k}(M(n)_C, \bar k)$. One easily checks that
\begin{equation}M(n)_C\simeq \frac {\bar k [C]}{(c-1)\bar kC}\simeq \bar kTr_C. \end{equation}  Hence we conclude that  $\mbox{Hom}_{\bar k}(M(n)_C, \bar k)$ and thus
$\mbox{Ext}^{n+1}_{\bar k [C]}(H^n(\mathcal{O}_X), \K))$ are of dimension one.   We also have an isomorphism
\begin{equation}
\mbox{Ext}^{n+1}_{\bar k [C]}(H^n(\mathcal{O}_X),\bar k)\simeq \hat H_{-1}(C, \mbox{Hom}_{\bar k}(H^n(\mathcal{O}_X), \bar k))=\mbox{Hom}_{\bar k}(H^n(\mathcal{O}_X), \bar k)_{Tr_C}/(c-1)(\mbox{Hom}_{\bar k}(H^n(\mathcal{O}_X), \bar k)). \end{equation}
Therefore we deduce that the map  
\begin{equation}
 \mbox{Ext}^{n+1}_{\bar k [C]}(H^n(\mathcal{O}_X),H^0(\mathcal{O}_X)) \rightarrow  \mbox{Ext}^{n+1}_{\bar k [C]}(H^n(\mathcal{O}_X)^G,H^0(\mathcal{O}_X))\end{equation} 
 is induced by the restriction homomorphism 
\begin{equation}\mbox{Hom}_{\bar k}(H^n(\mathcal{O}_X), \bar k)_{Tr_C}\rightarrow \mbox{Hom}_{\bar k}(H^n(\mathcal{O}_X)^G, \bar k)).\end{equation}
We note that for any $x\in  H^n(\mathcal{O}_X)$ there exists $x'$ such that $Tr_G(x)=Tr_C(x')$. This shows that
$Tr_G(H^n(\mathcal{O}_X))\subset  H^n(\mathcal{O}_X)^G\cap Tr_C(H^n(\mathcal{O}_X))$. Conversely, let  $\alpha=Tr_C(x)$ be an element of $H^n(\mathcal{O}_X)^G\cap Tr_C(H^n(\mathcal{O}_X))$ and let $\{g_i, 1\leq i\leq q\}$ be a set of representatives of $G/C$. We have the equalities
$$q\alpha=\sum_{1\leq i\leq q}g_iTr_C(x)=Tr_G(x).$$
This shows that $q\alpha$, and hence also $\alpha$, belongs to $Tr_G(H^n(\mathcal{O}_X))$. We conclude that
\begin{equation}Tr_G(H^n(\mathcal{O}_X))=  H^n(\mathcal{O}_X)^G\cap  Tr_C(H^n(\mathcal{O}_X)) .\end{equation}
It follows from (4.15) and (4.22 ) that the image of the map (4.21) is contained  in the image of $\mbox{Ext}^{n+1}_{\bar k [C]}(L',\bar k)$. Since $\mbox{Ext}^{n+1}_{\bar k [C]}(H^n(\mathcal{O}_X), \bar k)$ is of dimension one,  in order to complete the proof of the proposition,  it suffices to prove that the map (4.20) is not the zero map.  Let $f$ be a non-zero  element of
$ \mbox{Hom}_{\bar k}(H^n(\mathcal{O}_X)^G, \bar k))$,  trivial on $Tr_G(H^n(\mathcal{O}_X)) $, and let $x$ be an element of $H^n(\mathcal{O}_X)^G$ such that
$f(x)\neq 0$. In view of (4.22) we know that $x$ does not belong to  $Tr_C(H^n(\mathcal{O}_X))$. Let $V$ be a subvector space of $H^n(\mathcal{O}_X)$, containing $Tr_C(H^n(\mathcal{O}_X))$, and such that
$$H^n(\mathcal{O}_X)=V\oplus \bar k x .$$
Let $g$ be the element of  $ \mbox{Hom}_{\bar k}(H^n(\mathcal{O}_X), \bar k))$ defined by $g|V=0$ and $g(x)=f(x)$. It follows from (4.22) that the restriction of $g$ to
$Tr_G(H^n(\mathcal{O}_X))$ is trivial and therefore that $g|H^n(\mathcal{O}_X)^G=f$. This proves that,  as required, (4.20) is not the zero map. $\Box$

\section {Examples}
Our goal is to provide examples of smooth projective varieties of arbitrary large dimension, defined over an algebraically closed field of
characteristic $p$, endowed with the action of a cyclic group of order $p$,  which fulfills  Hypothesis 3.1.

Let $p> 2$ be a prime and let $\K$ be an algebraically closed field of
characteristic $p$. Define $G$ to be a cyclic group of order $p$
with generator $\sigma$. We fix an action of $G$ on the projective
space $\mathbb{P}^{p-1}_{\K}$ by letting $\sigma$ send the
point $x = (x_0:x_1:x_2\cdots:x_{p-1})$ to $\sigma(x) = (x_1:x_2: x_1:\cdots:x_0)$.

\begin{thm}
\label{thm:main}
Let $X$ be the closed subscheme of $\mathbb{P}_{\K}^{p-1}$ defined by the 
homogeneous polynomial $f_{p}=X_{0}...X_{p-1}+ X_{0}^{p-1}X_{1}+ X_{1}^{p-1}X_{2}+...+X_{p-1}^{p-1}X_{0}$.  Then $X$ 
 is a smooth irreducible hypersurface  of $ \mathbb{P}_{\K}^{p-1}$
of degree $p$ on which $G$ acts without fixed points.  The quotient morphism $X \to Y = X/G$ 
is \'etale, and  $Y$ is smooth, irreducible and 
projective of dimension $n = \mathrm{dim}(X) = p-2$. 
\end{thm}
\begin{proof}
Suppose that $f_{p}$ is reducible. Let $g$ be an irreducible, homogeneous 
divisor of $f_{p}$ of degree $m$ with $m<p$. We write $f_{p}=gh$. Since $\sigma(f_{p})=f_{p}$ then $\sigma(g)$ divides 
$f_{p}$ and thus  either $\sigma(g)$ divides $g$ or $\sigma(g)$ divides $h$. If 
$\sigma(g)$ divides $g$ there exists $\lambda \in \K^{*}$ such that $\sigma (g)=\lambda g$. Since $\sigma $ is of ordre $p$ we deduce that 
$\lambda^{p}=1$  and so  $\lambda=1$ and $\sigma(g)=g$. If $\sigma(g)\not\neq g$ then $\sigma (g)$ divides $h$ and therefore 
$g\sigma(g)...\sigma^{p-1}(g)$ divides $f_{p}$. Since $deg(f_{p})=p$ the degree of  $g$ has to be  $1$.  Therefore either  $\sigma(g)=g$ 
or  $g$ is homogeneous of degree $1$ and $f_{p}=g\sigma(g)...\sigma^{p-1}(g)$. For any polynomial $u$ we write  
$Tr(u)=\sum_{0\leq i\leq p-1}\sigma^{i}(u)$. Since there is no  monomial polynomial of degree $<p$  invariant by sigma we conclude 
that if  $g$ is of degree $<p$ and such that $\sigma(g)=g$  there exists some polynomial $u$ with  $g=Tr(u)$. Since $\K$ is of characteristic 
$p$ then  $g(1)=Tr(u)(1, ..., 1)=0$ which is impossible since $f_{p}(1, ...1)=1$. We now suppose that $f_{p}=g\sigma(g)...\sigma^{p-1}(g)$ with 
$g=a_{0}X_{0}+...+a_{p-1}X_{p-1}$. Since  $f_p$  doesn't contain any monomial   of the type $X_{i}^{p}$  at least 
one of the $a_{i}$'s has to be equal to $0$. Suppose for instance  that $a_{0}=0$. Then the coefficient of  $X_{0}^{p-1}X_{1}$ in $g\sigma(g)...\sigma^{p-1}(g)$  is  equal to $a_{1}^{2}a_{2}...a_{p-1}$. This implies that $a_{i}\not\neq 0$  for $1\leq i\leq p-1$ . Thus 
we will get in this product  the monomial polynomials $a_{1}a_{2}^{2}...a_{p-1}X_{0}^{p-1}X_{2},...,a_{1}a_{2}...a_{p-1}^2 X_{0}^{p-1}X_{p-1}$. 
Since $f_{p}$ doesn't contain such monomials  we conclude that $f_{p}$ can't be decomposed into such a product and thus that $f_{p}$  is irreducible. 

Now we claim   that it doesn't exist $x=(x_0, x_1, ..., x_{p-1})\neq (0, 0, ..., 0)$ such that 
\begin{equation} f_{p}(x)=f'_{p, X_{0}}(x)=...f'_{p, X_{p-1}}(x)=0 . \end{equation}
Suppose $x=(x_0, x_1, ..., x_{p-1})$ satisfies (5.1). We have
\begin{equation}f'_{p, X_{i}}(x)=x_0...\hat {x_i}...x_{p-1}+x_{\sigma^{-1}(i)}^{p-1}-x_{i}^{p-2}x_{\sigma(i)}=0, \ \ 0\leq i\leq p-1\ . \end{equation}
Using these  equations  one easily checks that if of the $x_{i}'s$ is equal to $0$ then the others are. 
 Let us denote by $P$ the product 
 $x_{0}...x_{p-1}$. It follows from (5.2)  that 
 \begin{equation}x_{i}f'_{p, X_{i}}(x)=P+x_ix_{\sigma^{-1}(i)}^{p-1}-x_{i}^{p-1}x_{\sigma(i)}=0, \ \ 0\leq i\leq p-1\ .\end{equation}We deduce from (5.3) the equalities
 \begin{equation}x_i^{p-1}x_{i+1}=iP+x_0^{p-1}x_1,  \ \ 0\leq i\leq p-1, \end{equation}
 (by convention we set $x_p=x_0$).  
 Using (5.4) the equation 
 $$f_p(x)=P+\sum_{0\leq i\leq p-1}x_i^{p-1}x_{i+1}=0$$
 can be written
 \begin{equation}f_p(x)=P+\frac{p(p-1)}{2}P+px_0^{p-1}x_1=P=0 .\end{equation}
 Then one (and thus all) of the $x_i$'s has to be equal to zero. This achieves the proof of the claim. 
 
 Since $(1: 1:...:1)$ is the unique  fixed closed point of $G$ acting on $\mathbb P$ we conclude that as required $X$ 
 is a smooth irreducible hypersurface  of $ \mathbb{P}_{\K}^{p-1}$
of degree $p$ on which $G$ acts without fixed points. Therefore  $\pi: X\to Y=X/G$ is a finite and \'etale morphism and $Y$ is a smooth, irreducible and projective variety over $\K$.  
 \end{proof}

Since $G$ is a cyclic $p$-group the $\K$-vector spaces $H^n(G, \K)$ and $H^{n+1}(G, \K)$ are  of dimension 1 as seen before. Thus,  to show  that Hypothesis 3.1 is satisfied for   
$(X, G)$,   it suffices to prove the following corollary. 
\begin{cor}
\label{cor:nicecor}
 The 
variety $X$ is of dimension $n$ and  

\begin{enumerate}

\item[i) ]$H^j(\O_X) = \{0\}$ if and only if $j \not \in \{0,n\}$. 

\item[ii)] If $\K$ is an algebraic closure of $\mathbb{Z}/p$ there is a model
$Y_0$ of $Y$ over a finite field $k_0$ such that the composition
of $X \to Y$ and $Y \to Y_0$ is a Galois morphism $X \to Y_0$. 
\end{enumerate} 
\end{cor}

\begin{proof}
For simplicity, we will write $\mathbb{P} = \mathbb{P}^{p-1}_{\K}$.  The ideal sheaf $I_X$
of $X$ is isomorphic to $\mathcal{O}_{\mathbb{P}}(-p)$ since
$X$ is a hypersurface of degree $p$.  We thus have an exact sequence
\begin{equation}
\label{eq:Hseq}
0 \to \mathcal{O}_{\mathbb{P}}(-p) \to \mathcal{O}_{\mathbb{P}} \to i_*\mathcal{O}_{X} \to 0
\end{equation} where $i$ is the closed immersion $X\to \mathbb{P}$. Let  $S_m$
be the $\K$-vector space of homogeneous degree $m$ polynomials in the homogeneous coordinate
functions $X_0,\ldots,X_{p-1}$ (if $m<0$ we set  $S_m=0$). 
>From \cite[Lemma 3.1]{Liu} we have, for any $m \in \mathbb{Z}$ 
$$H^j(\mathbb{P}, O_\mathbb{P}(m)) = 0 \quad \mathrm{if}\quad  j \neq 0, p-1$$
and 
$$H^0(\mathbb{P},O_{\mathbb{P}}(m)) = S_m \quad \mathrm{and}\quad  H^{p-1}(\mathbb{P}, O_\mathbb{P}(m)) =H^0(\mathbb{P},O_\mathbb{P}(-m-p))^\vee$$
where $L^\vee = \mathrm{Hom}(L,\K)$ if $L$ is a $\K$-vector space.  Using this in the long
exact cohomology sequence associated to (\ref{eq:Hseq}) shows part (i),  after noting,  for any integer $m$,  the equality 
$$H^m(X, \mathcal{O}_X)= H^m(\mathbb{P}, i_*(\mathcal{O}_X) .$$ 
 If $\K$ is an algebraic closure of $\mathbb{Z}/p$,
then there will be a model $Y_1$ of $Y$ over some finite extension $k_1$ of
$\mathbb{Z}/p$, and we can assume $k_1$ is the field of constants of $Y_1$.  Thus 
$Y = \K \times_{k_1} Y_1$ and the Frobenius automorphism $F_1$ of $\K$ over $k_1$
extends to a pro-generator of $\mathrm{Gal}(Y/Y_1)  = \mathrm{Gal}(\K/k_1)$. 
Embed $\K(X)$ into a separable  closure $\K(Y)^s$ of $\K(Y)$.  Here $\K(Y)^s$
is Galois over $k_1(Y_1)$, and a positive integral power $F_1^m$ of $F_1$ will lie in $\mathrm{Gal}(k(Y)^s/k(X))$. We now let $k_0$ be the fixed field of $F_1^m$ acting on $\K$ and we let $Y_0 = k_0 \otimes_{k_1} Y_1$ to have part (ii) of the Corollary.
\end{proof}

We end this section with an other example of scheme $X$ satisfying  the properties in Theorem \ref{thm:main}

\begin{prop} Suppose $X$ is an ordinary elliptic curve over $\K$.
Translation by the order $p > 2$ subgroup $X[p](\K)$ of $\K$-points of 
order $p$ of $X$ defines an \'etale action of the cyclic group $G = 
\mathbb{Z}/p$ on $X$ over $\K$.  Define $D$
to be the very ample degree $p$ divisor on $X$ which is the sum of the points in $X[p](\K)$.   There
is a basis $\{x_0,x_1,\ldots,x_{p-1}\}$ over $\K$ for $H^0(X,\mathcal{O}_X(D))$ which is cyclically
permuted by the action of a generator $\sigma$ for $G$.  Such a basis gives a  $G$-equivariant
closed immersion $X \to \mathbb{P}^{p-1}_{\K}$ with $G$ acting on $\mathbb{P}^{p-1}_{\K}$ by cyclically
permuting the homogeneous coordinates $(x_0:\ldots:x_{p-1})$  When $p = 3$, this defines
$X$ as a curve in $\mathbb{P}^2_{\K}$ having the properties in Theorem \ref{thm:main}.
\end{prop}

\begin{proof} Since $p \ge 3 = 2\cdot \mathrm{genus}(X)+1$ we see that $D$ is very ample
by \cite[Cor. IV.3.2]{Hartshorne}.  Since the action of $G$ on $X$ is \'etale, the equivariant Euler
characteristic $[H^0(X,\mathcal{O}_X(D)] - [H^1(X,\mathcal{O}_X(D)]$ in $G_0(\K[G])$ is in in the image
of the (injective) map $K_0(\K[G])\to G_0(\K[G])$.  Since $H^1(X,\mathcal{O}_X(D)) = 0$ by Riemann Roch Theorem,
we conclude that $H^0(X,\mathcal{O}_X(D))$ is a projective $kG$-module, which must be a free module
of rank one since $\# G = p$ and $\mathrm{dim}_{\K} H^0(X,\mathcal{O}_X(D)) = p$.  We now let
$x_0 \in H^0(X,\mathcal{O}_X(D))$
be a generator for $H^0(X,\mathcal{O}_X(D))$ as a free $\K G$-module and we define $x_i = \sigma^i x_0$ for $0 \le i \le p-1$.    
The Proposition is now clear.
\end{proof}

 \end{document}